\documentclass[12pt]{amsart}

\usepackage{geometry}
\usepackage[mathscr]{euscript}
\usepackage{ amssymb }
\usepackage[comma, numbers]{natbib}
\usepackage{bm}
\usepackage{enumerate}
\usepackage[english]{babel}
\usepackage{commath}
\usepackage{tikz-cd}
\usepackage{graphicx}
\usepackage [autostyle, english = american]{csquotes}
\MakeOuterQuote{"}

\theoremstyle{plain}
\newtheorem{theorem}{Theorem}[section]

\theoremstyle{plain}
\newtheorem{lemma}[theorem]{Lemma}

\theoremstyle{plain}
\newtheorem{corollary}[theorem]{Corollary}

\theoremstyle{definition}
\newtheorem{definition}[theorem]{Definition}

\theoremstyle{plain}
\newtheorem{proposition}[theorem]{Proposition}

\theoremstyle{remark}

\theoremstyle{definition}
\newtheorem{example}[theorem]{Example}

\theoremstyle{plain}

\theoremstyle{plain}

\theoremstyle{plain}

\begin{document}

\title{Abbott Dimension, Mathematics Inspired by \emph{Flatland}}
\markright{Abbott Dimension}
%\author{Jeremy Siegert}
\author{Jeremy Siegert}
\address{Ben Gurion University of the Negev, Beer Sheva, Israel} 
\email{siegertj@post.bgu.ac.il}
\date{\today}

\maketitle

\begin{abstract}
What is the "right way" to define dimension? Mathematicians working in the early and middle $20$th-century formalized three intuitive definitions of dimension that all turned out to be equivalent on separable metric spaces. But were these definitions the "right" ones? What would it mean to have the "right" definition of dimension? In this paper we attempt to inspire thought about these questions by introducing Abbott dimension, a geometrically intuitive definition of dimension based on Edwin Abbott's $1884$ novella \emph{Flatland}. We show that while Abbott dimension has intuitive appeal, it does not always agree with the classical definitions of dimension on separable metric spaces.
\end{abstract}

\section{Introduction.}
What is dimension? What does it mean for something to be three dimensional? Dimension seems to be a thing that we believe should be able to be described, but it is not at all obvious how to do so. The primary task of mathematicians working in the early $1900$'s on topics that are recognized today as being in the area known as dimension theory was to find a formal and geometrically intuitive definition for dimension, more specifically for the dimension of "nice" topological spaces.
\vspace{\baselineskip}

In $1905$ Poincar\'e gave one of the first notions of dimension to be formalized by mathematicians working in this area \cite{poincare}. In this philosophically oriented text, Poincar\'e (in more eloquent language than we will use) described dimension in terms of separation. Specifically, what makes a line one dimensional is that to cut it into disconnected pieces one must cut it along a subspace that is at least zero dimensional. Similarly, to cut the plane into disconnected pieces one must cut along a subspace that is at least one dimensional, and to cut a three dimensional space one must cut along something that is at least two dimensional, etc. A formal definition was not given, but the intuition is appealing and clear. One does not need any sophisticated knowledge of geometry or topology to appreciate it. Brouwer, Cech, Menger, and Urysohn would formalize this notion into a precise definition of dimension for "nice" topological spaces (for the sake of this account we can think of separable metric spaces). Chief among these were the small and large inductive dimensions $ind$ and $Ind$. These maintained the intuitive appeal of Poincar\'e's original description and made precise the meaning of the statement "$\mathbb{R}^{n}$ is $n$-dimensional". 
\vspace{\baselineskip}

Around the same time as Poincar\'e's writing, Lebesgue described the dimension of the cubes $[0,1]^{n}$ in a different, but still intuitive way. Namely, if one covers $[0,1]$ by open intervals one can always find a refinement of that covering that is also made up of open intervals such that every point of the interval is contained in at most two elements of the refinement. In the case of $[0,1]^{2}$, every covering by open squares can be refined by a covering of smaller open squares such that every point of $[0,1]^{2}$ is contained in at most three elements of the refinement. Almost twenty years later, Cech would give a general formulation for this idea applicable to more general spaces. We now call this definition the covering or topological dimension $dim$. Explorations of the definitions based on Poincar\'e's ideas and Lebesgue's ideas would continue throughout the twentieth century with one landmark result being that, on the class of separable metric spaces the dimension functions $ind$, $Ind$, and $dim$ agree. This history and the surrounding rich theory can be found in the excellent books by by Hurewicz and Wallman \cite{Hurewicz}, Engelking \cite{Engelking}, and Pears \cite{Pears}. 
\vspace{\baselineskip}

If one takes the primary task of early dimension theory to be to formalize a geometrically intuitive definition of the dimension of topological spaces (as we do), then we can safely say that the task was completed with great success. However, we can now reflect on this success and also reflect again on the concept of dimension and ask

\begin{displayquote}
\emph{Must every geometrically intuitive definition of dimension for nice topological spaces agree with the classical definitions?}
\end{displayquote}

\noindent
The question is not decisively answerable, but reflecting on it can provide some insights into our idea of dimension. If the answer is "yes" then we might safely believe that our geometrically intuitive definitions are the "right" definitions in the sense that any definition of dimension that seems acceptable to us on the level of geometric intuition must agree with the classical definitions on nice topological spaces. However, if the answer is "no", that is if there is a space for which we think there ought to be a clearly definable dimension and two different intuitive definitions of dimension disagree on that space, then what does that say about our ideas of what dimension is? We will ultimately challenge the reader to attempt to answer this question for themselves, but our main objective in this paper is to argue that the answer to our question is "no".
\vspace{\baselineskip}

We do this by providing a geometrically intuitive definition for the dimension of Hausdorff spaces based on a conception of dimension that predates that of both Lebesgue and Poincar\'e. It appears in the $1884$ novella \emph{Flatland} by Edwin Abbott. \emph{Flatland} is the story of an inhabitant (a square) of the eponymous two dimensional land and his experiences with realms of both higher and lower dimensions. The novella is well known for its presentation of dimension. Surprisingly, however, a rigorous treatment of the ideas does not seem to have been done. In contrast to the conceptions of dimension by Poincar\'e and Lesbegue which were based on separation and covers, Abbott's notion of dimension is based on "sight". More specifically, to reside in a higher dimensional space one must be able to "see inside of" lower dimensional spaces.
\vspace{\baselineskip}

 We attempt to provide a rigorous formulation of this idea for Hausdorff spaces and call the resulting topological invariant Abbott dimension, denoted $Ab(X)$ for a space $X$. In section $2$ we will review some basic preliminary definitions and results, as well as describe the bare minimum of properties that a definition of dimension ought to satisfy. In section $3$ we will define Abbott dimension. Part of this section will include an explanation of the notion of the aforementioned "sight" as seen in \emph{Flatland} and how it is formalized in our definition. Also in section $2$ are the primary basic properties of $Ab$, such as topological invariance and the fact that $Ab(\mathbb{R}^{n})=n$. In section $4$ we show that $Ab$ does not agree with the three classical definitions of dimension on separable metric spaces (in fact compact metric spaces) by showing that the Abbott dimension of all of what are called hereditarily indecomposable continua is $1$. Along the way we give some detailed descriptions of some well-known examples. We conclude in section $5$ with a brief discussion and a challenge to the reader. An appendix containing some technical discussions and arguments is included after the references. Throughout this paper we will use some basic concepts from point set topology without defining them. In particular we will use terms such as Hausdorff, normal, and compact. A reader unfamiliar with these terms may still read this paper and understand its main ideas. The inexperienced reader should regard these terms and concepts as "legal technicalities" that must be met to make mathematical statements rigorously provable.

\section{Preliminaries.}

In this section we will collect the basic definitions and results that will be needed to precisely state and prove relationships between Abbott dimension and classical definitions of dimension. We will also use this section to describe the large inductive dimension (which is one of the classical definitions of dimension) that we will compare to our definition of Abbott dimension. The large inductive dimension reifies the notion of dimension by Poincar\'e that is characterized by separation. That is, it makes precise the idea that an $n$-dimensional space is one that requires an $(n-1)$-dimensional subspace to be removed in order to become disconnected. We will ultimately use the extant results we state here for the large inductive dimension to show that the Abbott dimension (which we have not yet defined) of $\mathbb{R}^{n}$ is $n$. A thorough treatment of the large inductive dimension can be found in \cite{Engelking}. A reader comfortable with classical dimension theory may safely forego this section. For a subspace $Y$ of a topological space $X$ we denote by $cl_{X}(Y)$ the closure of $Y$ in $X$ and by $\partial_{X}Y$ the boundary of $Y$ in $X$.

\begin{definition}
By a {\bf continuum} we mean a (not necessarily metric) compact, connected, Hausdorff Space. In the event that such a space is a single point we call it a {\bf degenerate continuum}.
\end{definition}

\begin{definition}
Given a topological space $X$ and two disjoint closed subsets $A,B\subseteq X$, a {\bf separator} in $X$ between $A$ and $B$ is a closed set $C$ that is disjoint from $A\cup B$ and is such that $X\setminus C=U\cup V$ where $U$ and $V$ are disjoint open sets that contain $A$ and $B$, respectively. 
\end{definition}

The above definition makes precise what it means to cut a space between two regions therein. In the following definition it is hopefully apparent that with a different notion of "separator" one may obtain a different definition of dimension based on the idea of separation\footnote{Alternative formulations of what a separator could be do exist. Indeed, in what was perhaps the first precise formulation of dimension, Brouwer in \cite{brouwer} defined his invariant, the dimensiongrad, of a space. This invariant is identical in definition to the large inductive dimension we have defined with the difference being that instead of using separators it uses the notion of a \emph{cut}. A cut in a space $X$ between disjoint closed subsets $A$ and $B$ is a third closed set $C\subseteq X$ that is disjoint from $A$ and $B$ and is such that any continuum $K\subseteq X$ that intersects $A$ and $B$ must also intersect $C$. This definition of dimension agrees with the classical definitions mentioned in the introduction on the class of compact metrizable spaces (see \cite{fedorchuk}).}.

\begin{definition}
The {\bf large inductive dimension} of a normal space $X$, denoted $Ind(X)$, is defined inductively in the following way:\\
\begin{enumerate}
\item $Ind(X)=-1$ if and only if $X=\emptyset$.
\item $Ind(X)\leq n$ for $n\geq 0$ if and only if for every pair of disjoint closed sets $A,B\subseteq X$ there is a separator $C\subseteq X$ between $A$ and $B$ that satisfies $Ind(C)\leq n-1$. 
\end{enumerate}
\noindent
The value $Ind(X)$ is then defined to be the least $n$ for which $Ind(X)\leq n$ is true. In the event that $Ind(X)\not\leq n$ for each $n$ we say that $Ind(X)=\infty$. 
\end{definition}

In the Introduction we briefly mentioned that there are some properties that a definition of dimension absolutely must satisfy to be acceptable as a suitable definition. Here we state those properties in the context of the large inductive dimension. We will make use of these properties of $Ind$ to show the analogous properties for Abbott dimension in the next section.

\begin{theorem}[Topological Invariance]\label{invarianceInd}
If $X$ and $Y$ are normal spaces that are homeomorphic, then $Ind(X)=Ind(Y)$. 
\end{theorem}

Topological invariance of dimension reflects the intuition we have that continuous deformations should not alter dimension.

\begin{theorem}[The Subspace Theorem]\label{subspaceInd}
If $X$ is a separable metric space and $Y\subseteq X$ is any subspace, then $Ind(Y)\leq Ind(X)$. 
\end{theorem}

Certainly pulling a higher dimensional subspace out of a lower dimensional one should not be possible.

\begin{theorem}[Euclidean Dimension]\label{euclideanInd}
For each $n$, $Ind(\mathbb{R}^{n})=n$.
\end{theorem}

This is perhaps the most important property that a definition of dimension must satisfy. Indeed, any property of topological spaces that does not meet this requirement cannot be reasonably called dimension.
\vspace{\baselineskip}

In addition to those stated above, we will make use of two additional properties of $Ind$:

\begin{theorem}\label{dim0}
If a separable metric space $X$ is such that $Ind(X)=0$, then $X$ is totally disconnected. 
\end{theorem}

\begin{theorem}\label{euclideanInd2}
For every $K\subseteq\mathbb{R}^{n}$ we have that $Ind(K)=n$ if and only if $K$ has nonempty interior in $\mathbb{R}^{n}$. 
\end{theorem}

Computing the large inductive dimension of a space can be quite difficult, even for spaces that are well behaved. In particular, establishing Theorem \ref{euclideanInd} is quite difficult. However, once theorems such as Theorem \ref{euclideanInd2} have been established, computing the large inductive dimension of some Euclidean subspaces becomes much easier, as is the case for the "Sierpinski carpet", described below.

\begin{example}
Let $F_{0}$ be the subspace $[0,1]\times[0,1]$ of $\mathbb{R}^{2}$. Obtain $F_{1}$ from $F_{0}$ by dividing $F_{1}$ into the nine congruent squares $\left[\frac{i}{3},\frac{i+1}{3}\right]\times\left[\frac{j}{3},\frac{j+1}{3}\right]$ for $0\leq i,j\leq 2$ and removing the "center square", $\left[\frac{1}{3}.\frac{2}{3}\right]\times\left[\frac{1}{3},\frac{2}{3}\right]$. Then $F_{1}$ is a union of eight identical squares in $\mathbb{R}^{3}$. Obtain $F_{2}$ by repeating this process of deleting the center square from each of the remaining squares in $F_{1}$. In like fashion, for each $n\in\mathbb{N}$ we can construct a subspace $F_{n}\subseteq F_{n-1}$ by this process of deletion. The "Sierpinski carpet" is the limit space $X=\bigcap_{n=0}^{\infty}F_{n}$ with the subspace metric inherited from $\mathbb{R}^{2}$. It is easy to see that $X$ is a connected space that does not contain any open subset of $\mathbb{R}^{2}$ and so by Theorem \ref{dim0} and Theorem \ref{euclideanInd2} we have that $Ind(X)=1$. What's more, $X$ is what is called a universal space for separable metric spaces of large inductive dimension $1$ which means that every separable metric space $Y$ with $Ind(Y)=1$ embeds into $X$. Four stages of the construction of the Sierpinski carpet are shown below. The space and its universality were originally discussed by Sierpinski in \cite{sierpinski}. An English version of the construction and universality result can be found in \cite{Engelking}.
\end{example}

\begin{figure}[!h]
\centering
\includegraphics{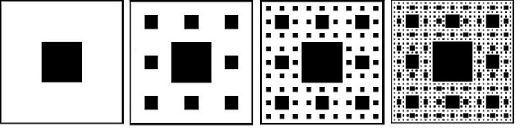}
\caption{Stages $F_{1},F_{2},F_{3},$ and $F_{4}$ in the construction of the Sierpisnki carpet.}
\end{figure}

\newpage
\section{Abbott Dimension.}
In this section we will define the Abbott dimension of Hausdorff spaces. We will show that this property is a topological invariant, that it satisfies the subspace theorem (the analog of Theorem \ref{subspaceInd} for Abbott dimension), and most importantly, we will show that the Abbott dimension of $\mathbb{R}^{n}$ is $n$. However, before doing any of this some discussion of the inspiration for our definition, Edwin Abbott's novella \emph{Flatland}, is warranted. 
\vspace{\baselineskip}

The story of \emph{Flatland} is broken up into two parts. The first part is an explanation of the eponymous country Flatland, a two dimensional world resembling $\mathbb{R}^{2}$, as told by the protagonist, a square. The society of Flatland is designed to be a satirical take on the culture of England at the time (the Victorian era). The second half of the story is the part that is most relevant for us. In this part the square dreams of a visit to the zero dimensional Pointland and the one dimensional Lineland. Between these dreams the square is visited by a sphere from Spaceland, a three dimensional world resembling $\mathbb{R}^{3}$, who to the square appears as a circle of varying diameter. The sphere takes the square on a journey to three dimensional Spaceland. Throughout the text one finds that the notion of being from a "higher dimension" means that you can see "inside" the bodies of lower dimensional beings. When the square has a vision of Lineland this amounts to the square being able to see the "internal organs" of the denizens of Lineland.

\begin{displayquote}

\emph{Though he had heard my voice when I first addressed him, the 
sounds had come to him in a manner so contrary to his experience that he had made no answer, "seeing 
no man,'' as he expressed it, "and hearing a voice as if it were from my own intestines.'' Until the moment 
when I placed my mouth in his World, he had neither seen me, nor heard anything except confused - 
sounds beating against - what I called his side, but what he called his inside or stomach; nor had he even 
now the least conception of the region from which I had come.\\
\null\hfill-The square describing the king of Lineland (pages 116-117 of \cite{Abbott})}
\end{displayquote}

\noindent
Similarly, when the sphere brings the square into three dimensional space this notion of dimension allows the square to see "over" the walls of the mines and caverns of Flatland.

\begin{displayquote}
\emph{Once more I felt myself rising through space. It was even as the Sphere had said. The further we receded 
from the object we beheld, the larger became the field of vision. My native city, with the interior of every 
house and every creature therein, lay open to my view in miniature. We mounted higher, and lo, the 
secrets of the earth, the depths of mines and inmost caverns of the hills, were bared before me.\\
\null\hfill-The square looking on Flatland from a three dimensional perspective (page 140 of \cite{Abbott})}
\end{displayquote}

\noindent
One can find other quotes that display the idea of being able to "see inside" things of lower dimensions from a higher dimensional perspective. The task of defining dimension using this idea is then to formalize this notion of "sight". We do so in the following way. Throughout the remainder of this paper all spaces are assumed to be Hausdorff. For points $x,y$ in a space $X$, not necessarily distinct, let $\mathcal{C}(x,y,X)$ be the collection of open connected subsets of $X$ that contain both $x$ and $y$.

\begin{definition}\label{lineofsightdef}
Let $X$ be a nonempty Hausdorff space, $x,z\in X$ distinct points, and $K\subseteq (X\setminus\{z\})$ a subspace that contains $x$. A continuum $C\subseteq X$ containing $x$ and $z$ is called a {\bf line of sight} from $z$ to $x$ if the following hold:\\
\begin{enumerate}
\item $\mathcal{C}(x,z,X)$ contains a neighborhood basis for $C$.
\item For all open sets $U\subseteq K$ containing $x$, there is an element $V_{U}\in\mathcal{C}(x,z,X)$ that contains $C$ and satisfies $V_{U}\cap K\subseteq cl_{X}(V_{U})\cap K\subseteq U$.
\end{enumerate}

\end{definition}

\begin{figure}[!h]
\centering
\includegraphics{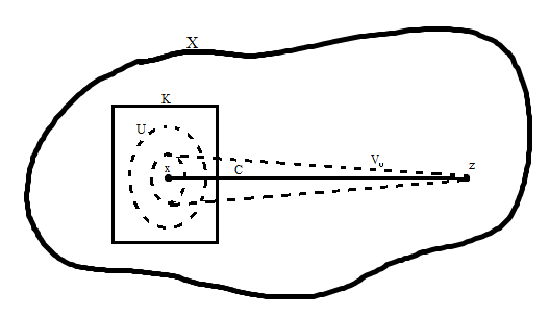}
\caption{A line of sight from $z$ to $x\in K$.}
\end{figure}

This is our formalization of what it means for a point in one space to be visible from a point outside of that space. One should think of the point $z$ in the above Definition as being the point from which the square looks on the interiors of the denizens of Lineland. Sticking with this analogy, the space $K$ is Lineland itself. The two conditions of the definition are meant to make precise the idea that we can "narrow our gaze" at the expense of visibility of the surrounding area. One may imagine looking at a spot on the wall. By squinting, one may focus their attention on the spot at the expense of being able to see some of the area at the periphery of our vision. Likewise, to do this focusing we do not need to change the line of sight from our eyes to the spot. With our definition we mean to formalize a point $x$ being \emph{directly} visible from $z$, as is conveyed in the figure above.

\begin{definition}\label{main definition}
Let $X$ be a nonempty Hausdorff space. We define the {\bf Abbott dimension} of $X$, denoted $Ab(X)$, in the following way. Let $X_{0}$ denote the power set of $X$. Assuming $X_{n}$ has been defined for $n\geq 0$ define $X_{n+1}$ to be the collection of subsets $Y\subseteq X$ for which there is some $K\in X_{n}$ such that:
\begin{enumerate}
\item $K\subseteq Y$.
\item There is a point $z\in Y\setminus K$ such that for all points $x\in K$ there is a line of sight $W_{z,x}\subseteq Y$ from $z$ to $x$ such that:

\begin{itemize}

\item There is a $V\in\mathcal{C}(x,z,Y)$ containing $W_{z,x}$ such that if $V^{\prime}\in\mathcal{C}(x,z,Y)$ contains $W_{z,x}$ and is contained in $V$, then there is a $D\subseteq\partial_{Y}(V^{\prime})$ such that $D\in X_{n}$.

\end{itemize}

\end{enumerate}
Then $Ab(X)$ is defined to be the greatest $n$ for which $X_{n}\neq\emptyset$. If $X_{n}\neq\emptyset$ for all $n$, then $Ab(X)$ is said to be infinite. Finally, we define $Ab(\emptyset)=-1$. 
\end{definition}

 \noindent
In the above definition if $Y_{n}\in X_{n}$ and $Y_{n-1}\in X_{n-1}$ is such that $Y_{n-1}$ can be used to show that $Y_{n}\in X_{n}$, then we say that $Y_{n-1}$ {\bf witnesses} that $Y_{n}\in X_{n}$, and denote this $Y_{n-1}\ll_{n} Y_{n}$. We also call the point $z$ used in item $(2)$ an {\bf observation point} for $K$.
\vspace{\baselineskip}

This definition says that if a space $X$ contains an $n$ dimensional subspace that it can "see inside of" from some higher dimensional perspective, then it must be at least $n+1$ dimensional.  
\vspace{\baselineskip}

\begin{example}
The Abbott dimension of a singleton is $0$. Similarly, the Abbott dimension of any totally disconnected space is $0$. Recall that a space is totally disconnected if the only non-empty connected subspaces are points. Let $X$ be totally disconnected. If $Y\subseteq X$ is such that $Y\in X_{n}$ for some $n\geq 1$, then there is a subspace $K\subseteq Y$, a point $x\in K$, a point $z\in Y\setminus K$, and a continuum $W_{z,x}\subseteq Y$ that contains both $x$ and $z$. However, the only connected subspaces of $X$, and consequently $Y$, are points. Therefore no such $W_{z,x}$ can exist and $X_{n}$ must be empty for all $n\geq 1$. This yields $Ab(X)=0$. 
\end{example}

A useful obvious fact is the following.

\begin{lemma}\label{descending sets}
Given a space and $k\in\mathbb{N}$ we have that $X_{k+1}\subseteq X_{k}$. 
\end{lemma}

The definition of $Ab$, while inductive in nature, takes on a different flavor than the inductive dimensions $Ind$ and $ind$ in that, with the latter two dimension functions it is often easier to provide upper bounds for the dimension of space. However, with $Ab$ the opposite is true. For a space $X$, establishing $Ab(X)\geq n$ only requires finding a nested sequence of $n+1$ subspaces $Y_{0}\lll_{1}Y_{1}\ll_{2}\cdots\ll_{n}Y_{n}\subseteq X$. Indeed, for $\mathbb{R}^{3}$ one may use a point, then a line segment containing that point, then a square one of whose sides is the segment, and finally a cube having that square as a side to show that $Ab(\mathbb{R}^{3})\geq 3$. In general we have,

\begin{proposition}\label{lower bound}
For every $n\geq 1$, $Ab(\mathbb{R}^{n})\geq n$. 
\end{proposition}

The next two results are, at a glance, seemingly obvious. The first, Lemma \ref{subspaces witness}, says that if one can see inside of of a subspace $Y\subseteq X$ from an observation point $z\in X\setminus Y$ and $K\subseteq Y$, then there should also be an observation point in $X\setminus K$ from which one can see inside of $K$. Said perhaps more simply, if you draw a square on the wall, you can see the interior of the square, and you can see the interior of any of the line segments that make up the sides of the square. The second result, Proposition \ref{intermediate dimensions}, says that a space of finite dimension should contain subspaces of every possible smaller dimension. Again, this is likely obvious primarily because of our intuition regarding dimension. However, even in the case of the classical definitions of dimension a nontrivial proof is required.\footnote{Interestingly, for the classical definitions of dimension the result does not hold if the space is infinite dimensional. In \cite{henderson} a compact metric space of infinite dimension (with respect to any of the classical definitions) is constructed such that every subspace is either also infinite dimensional, or zero dimensional.}

\begin{lemma}\label{subspaces witness}
Let $X$ be a space with $Ab(X)=n$ and $Y\subseteq X$ be such that $Y\ll_{n} X$. If $K\subseteq Y$ with $K\in X_{k}$ for some $k\leq n-1$, then $K\ll_{k+1} X$.

\end{lemma}
\begin{proof}
We will prove that $K\ll_{k+1} X$ by definition. Let $z\in X\setminus Y$ be an observation point for $Y$ and let $x\in K$ be given. Because $K\subseteq Y$ and $Y\ll_{n} X$ there is a continuum $W_{z,x}\subseteq X$ that contains both $x$ and $z$. Also $W_{z,x}$ can be chosen so that $\mathcal{C}(x,z,X)$ contains a neighborhood basis for $W_{z,x}$. Let $U\subseteq K$ be an open subset of $K$ that contains $x$. Then $U=U^{\prime}\cap K$ for some subset $U^{\prime}\subseteq Y$ that is open in $Y$. Then, there is a $V_{U^{\prime}}\in\mathcal{C}(x,z,X)$ containing $W_{z,x}$ and satisfying $V_{U^{\prime}}\cap Y\subseteq cl_{X}(V_{U^{\prime}})\cap Y\subseteq U^{\prime}$. Intersecting all of these with $K$ yields

\[V_{U^{\prime}}\cap K\subseteq cl_{X}(V_{U^{\prime}})\cap K\subseteq U^{\prime}\cap K=U\]

\noindent
which establishes item $(2)$ of the Definition of a line of sight. The bullet point of item $(2)$ of the definition for Abbott dimension holds because it is a property of $W_{z,x}$ that is independent of $Y$ and $K$. We have then shown that $X\in X_{k+1}$ is witnessed by $K$. That is, $K\ll_{k+1} X$.

\end{proof}

\begin{proposition}\label{intermediate dimensions}
If $X$ is a space with $Ab(X)=n\geq 0$, then for all $0\leq k\leq n$ there is a $Y\subseteq X$ such that $Ab(Y)=k$. 
\end{proposition}
\begin{proof}
If $n=0$ then the result is trivial. Otherwise assume that $Ab(X)=n>0$ and assume that $k$ is such that $1\leq k\leq n-1$ and there is no $Y\subseteq X$ satisfying $Ab(Y)=k$. This means that if $Y\subseteq X$ is any nonempty subset, then either $Ab(Y)<k$ or there is a $K\subseteq Y$ such that $K\in X_{k+1}$. As $Ab(X)=n>k$ we have that $X_{n}$ must not be empty, which by Lemma \ref{descending sets} gives us that $X_{k+i}$ is nonempty for $0\leq i\leq n-k$. We claim that $X_{k+1}=X_{k+i}$ for all $1\leq i$, which will contradict $Ab(X)$ being finite. If $Y\subseteq X$ is such that $Y\in X_{k+1}$, then there is some $K\subseteq Y$ such that $K\in X_{k}$ and $K\ll_{k+1}Y$. By assumption we have that $Ab(K)>k$ so there must be some $K^{\prime}\subseteq K$ such that $K^{\prime}\in X_{k+1}$. As $K\ll_{k+1}Y$ we have by Lemma \ref{descending sets} that $K^{\prime}\in X_{k}$. Moreover, by Lemma \ref{subspaces witness} we have that $K^{\prime}\ll_{k+1}Y$. After a moments consideration of our assumption that every element of $X_{k}$ has a subset contained in $X_{k+1}$ and item $2$ of the Definition of Abbott dimension, we arrive at the fact that $K^{\prime}\ll_{k+2}Y$, showing that $Y\in X_{k+2}$. As $X_{k+2}\subseteq X_{k+1}$, this proves their equality. Now assume that $X_{k+1}=X_{k+j}$ for $1\leq j$ and let $Y\in X_{k+j}$ and $K\in X_{k+j-1}$ be such that $K\ll_{k+j}Y$. From the assumption that $X_{k+j-1}=X_{k+j}$ one can quickly deduce that, in fact, $K\ll_{k+j+1}Y$, yielding $Y\in X_{k+j+1}$. By induction and Lemma \ref{descending sets} we then have that $X_{k+1}=X_{k+i}$ for all $i\geq 1$. This however implies that $Ab(X)=\infty$, contradicting our assumption that $Ab(X)$ is finite. Therefore there must be some $Y\subseteq X$ such that $Ab(Y)=k$. 

\end{proof}

\noindent
The next two results are easy consequences of the definition of Abbott dimension, but are none the less part of the bare minimum of the properties a dimension function ought to satisfy.

\begin{theorem}\label{subspace theorem}
If $Y$ is a subspace of a space $X$, then $Ab(Y)\leq Ab(X)$.
\end{theorem}
\begin{proof}
Let $X$ be a space and $Y\subseteq X$ be given. Let $K_{1},K_{2}\subseteq Y$ be such that $K_{1}\ll_{n}K_{2}$. Then $K_{1}\in Y_{n-1}$ and $K_{2}\in Y_{n}$. Continua in $K_{2}$ are continua regardless of whether or not $K_{2}$ is considered as a subspace of $Y$ or $X$. Similarly, open connected subsets of $K_{2}$ are open connected subsets of $K_{2}$ regardless of whether $K_{2}$ is considered as a subset of $Y$ or $X$. Therefore $K_{2}\in X_{n}$ and $K_{1}\in X_{n-1}$. Moreover $K_{1}\ll_{n}K_{2}$ as subsets of $X$. Therefore if $n$ is the greatest natural number for which $Y_{n}\neq\emptyset$ then $X_{n}\neq\emptyset$ as well. Therefore $Ab(Y)\leq Ab(X)$. 
\end{proof}

\begin{theorem}\label{invariance}
If $X$ and $Y$ are homeomorphic spaces, then $Ab(X)=Ab(Y)$. 
\end{theorem}
\begin{proof}
This result is immediate upon noting that homeomorphisms are continuous bijections that preserve open connected sets, continua, and neighborhood bases.
\end{proof}

What remains to be shown for Abbott dimension is that it satisfies the most critical property of a dimension function, namely that the dimension of $\mathbb{R}^{n}$ is $n$. The proof itself is rather technical and we relegate the details of the argument to the appendix. However, we can give a rough overview of the argument here. To prove that $Ab(\mathbb{R}^{n})=n$ we make use of Proposition \ref{lower bound} to show that Abbott dimension is bounded above by the large inductive dimension on a class of separable metric spaces that includes $\mathbb{R}^{n}$ (also proved in the appendix). As is fairly common for results of this type, the cases of finite and infinite dimension are treated separately. In the finite case one uses induction on the Abbott dimension of the space and then assumes that the large inductive dimension is smaller than the Abbott dimension. In the infinite case one assumes that the space has infinite Abbott dimension and then uses induction on the large inductive dimension, showing that one arrives at a contradiction regardless of which finite value the large inductive dimension is assumed to be. As the definition of both Abbott dimension and the large inductive dimension ultimately involve the boundaries of open sets having smaller dimension, one can arrive at a contradiction for both the finite and infinite case without too much trouble. Again, the interested reader can find the technical details in the appendix.

\begin{theorem}\label{euclidean space}
For all $n \geq 1$, $Ab(\mathbb{R}^{n})=n$. 
\end{theorem}

With the results of this section we have established Abbott dimension as a legitimate definition of dimension. As promised in the introduction, our next task is to show that Abbott dimension does not agree with the classical definitions of dimension (represented here by $Ind$) on the class of separable metric spaces. 

\section{Inequivalence of Abbott dimension with classical definitions.}
In this section we will show that Abbott dimension does not agree with the classical definitions of dimension. Specifically we will begin by showing that the Abbott dimension of the Knaster-Kuratowski fan (a noncompact space) is $0$ while its large inductive dimension is $1$. Then we will show that the Abbott dimension of hereditarily indecomposable continua is at most $1$. However, as shown by Bing in \cite{Bing}, there are hereditarily indecomposable continua of arbitrarily high covering dimension (and consequently large inductive and small inductive dimension). Along the way we will give a detailed description of some well-known spaces in the class of indecomposable and hereditarily indecomposable continua. The reader comfortable with these kinds of spaces may skip from the definition of indecomposable and hereditarily indecomposable continua to Theorem \ref{hindecomposableAbbott}, the main result of this section.
\vspace{\baselineskip}

\noindent
As said above, we will begin with the Knaster-Kuratowski fan. Its construction is as follows (and as found in \cite{Counterexamples}):
\vspace{\baselineskip}

Let $C$ denote the one thirds Cantor set in $[0,1]\times\{0\}\subseteq\mathbb{R}^{2}$ and let $z=\left(\frac{1}{2},\frac{1}{2}\right)$. Then let $Y$ be the union of all line segments joining $z$ to points of $C$. For a particular $x\in C$ let $L_{x}$ be the line segment in $Y$ joining $x$ to $z$. Now let $C_{1}\subseteq C$ be the set of points that appear as endpoints of intervals removed in the construction of $C$, and let $C_{2}$ be the complement of $C_{1}$ in $C$. If $x\in C_{1}$ define $\hat{L}_{x}=\{(r,s)\mid s\in\mathbb{Q}\}$. Similarly, if $x\in C_{2}$ define $\hat{L}_{x}=\{(r,s)\mid s\notin\mathbb{Q}\}$. The Knaster-Kuratowski fan is the space $X=\bigcup_{x\in C}\hat{L}_{x}$. Why this space is referred to as a "fan" is more evident when an image of the space is rendered. The fan is well known for being a space that is connected, but not locally connected. It also has what is known as a dispersion point. That is, the point $z=\left(\frac{1}{2},\frac{1}{2}\right)$ is such that $X\setminus\{z\}$ is totally disconnected.

\begin{figure}[!h]
\centering
\includegraphics{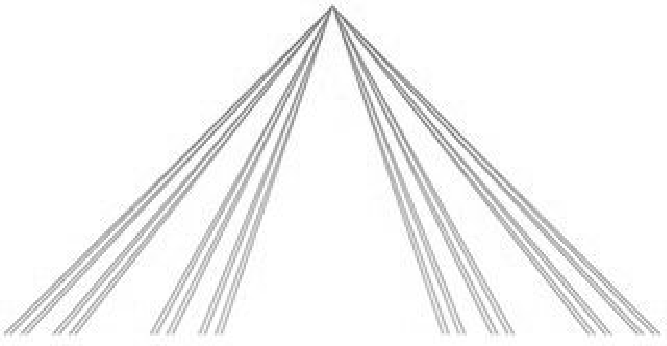}
\caption{Knaster-Kuratowski fan.}
\end{figure}

\begin{theorem}
If $X$ is the Knaster Kuratowski fan, then $Ab(X)=0$ and $Ind(X)=1$. 
\end{theorem}
\begin{proof}
As $X$ is nonempty we have that $Ab(X)\geq 0$. However, as $X$ does not contain any nontrivial continua we have that $Ab(X)<1$. Therefore $Ab(X)=0$. However, as $X$ is connected we must have that $Ind(X)\geq 1$. Meanwhile $X$ has empty interior in $\mathbb{R}^{2}$ and therefore $Ind(X)<2$, so $Ind(X)=1$. 
\end{proof}

With the Knaster Kuratowski fan we have already shown that Abbott dimension does not agree with the classical definitions on separable metric spaces. However, leaving the story here would be somewhat unsatisfying. The Knaster Kuratowski fan, while separable, is not compact and feels like it could be just an oddity. However, as we will now proceed to show, enforcing a requirement of compactness does nothing to insure any kind of coincidence between Abbott dimension and the classical definitions. More specifically, we will be considering the family of hereditarily indecomposable continua, defined as follows.

\begin{definition}\label{indecomposable}
A metric continuum $K$ is called {\bf indecomposable} if $K$ cannot be written as the union of two proper subcontinua. We say that $K$ is {\bf hereditarily indecomposable} if every subcontinuum of $K$ is also indecomposable.
\end{definition}

To help digest this definition let's consider the following example, the "bucket handle" continuum defined as follows and as can be found in \cite{Ingram}:
\vspace{\baselineskip}

Let $C$ be the one thirds Cantor set in $[0,1]\times\{0\}\subseteq\mathbb{R}^{2}$. Define $C_{0}$ to be the collection of semicircles lying above the $x$-axis in $\mathbb{R}^{2}$ that have endpoints in $C$ and are centered on $\left(\frac{1}{2},0\right)$. Then define $C_{1}$ to be the collection of semicircles lying below the $x$-axis centered on the midpoint of $\left[\frac{2}{3},1\right]$ and with endpoints in $C\cap\left[\frac{2}{3},1\right]$. Further collections of semicircles $C_{i}$ are defined to be the collections of semicircles lying below the $x$-axis that are centered on the midpoint of the interval $\left[\frac{2}{3^{i}},\frac{3}{3^{i}}\right]$ with endpoints in $C\cap\left[\frac{2}{3^{i}},\frac{3}{3^{i}}\right]$. The bucket handle continuum is the union of all of the collections $C_{i}$ for $i\geq 0$.

\begin{figure}[!h]
\includegraphics{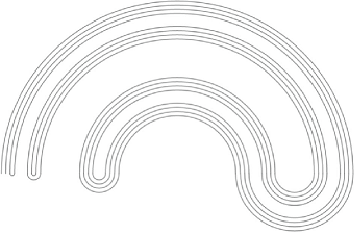}
\centering
\caption{An approximate image of the bucket handle continuum.}
\end{figure}

There are other examples of interesting indecomposable continua, such as the "Lakes of Wada" which is an indecomposable continuum that fascinatingly arises as the common boundary of three open disjoint connected subsets of $\mathbb{R}^{2}$. However, while it would be easy for us to spend a great amount of time working through such examples, they are not the focus of this section. We will direct those curious about the Lakes of Wada to its original description in \cite{Yoneyama} (on page $60$). As stated in the introduction, the subfamily of indecomposable continua that is of primary interest to us in this section is hereditarily indecomposable continua which, as described in Definition \ref{indecomposable}, are indecomposable continua all of whose subcontinua are themselves indecomposable. Perhaps the most famous example of a hereditarily indecomposable continuum is the "pseudoarc".\footnote{In 1920 Knaster and Kuratowski asked in \cite{knasterkuratowski} if a nondegenerate homogeneous continuum in the plane must be a closed curve. The next year Mazurkiewicz asked in \cite{mazurkiewicz} if every continuum in the plane which is homeomorphic to all of its nondegenerate subcontinua must be an arc. In \cite{knaster} Knaster would give an example of a hereditarily indecomposable continuum in 1923. In \cite{bing48} Bing would would answer Knaster and Kuratowski's question negatively with his own construction in 1948, and Moise would answer Mazurkiewicz's question negatively in \cite{moise}. Moise dubbed his example the "pseudoarc" due to its similarity to an arc. Bing would go on in \cite{bing51} to show that his, Knaster's, and Moise's examples were all homeomorphic. This history and the results surrounding the pseudoarc can be found in \cite{nadler}.} Below we give a construction of the pseudoarc as in \cite{nadler}. This construction is not needed for the results of this section, but we include it here simply for the sake of completion. The reader interested only in the results of this section may safely skip it.
\vspace{\baselineskip}

Call a finite collection $\{U_{1},\ldots,U_{n}\}$ of subsets of a set $X$ a {\bf simple chain} from points $x$ to $z$ in $X$ if $U_{i}\cap U_{j}\neq\emptyset$ if and only if $|i-j|\leq 1$, $x$ is only in $U_{1}$, and $z$ is only in $U_{n}$. Elements of simple chains are appropriately called {\bf links}. A simple chain $\mathcal{V}=\{V_{1},\ldots,V_{m}\}$ is a {\bf subchain} of another simple chain $\mathcal{U}=\{U_{1},\ldots,U_{n}\}$ when $\mathcal{V}\subseteq\mathcal{U}$. If there are simple chains $\mathcal{U}=\{U_{1},\ldots,U_{n}\}$, $\mathcal{V}=\{V_{1},\ldots,V_{m}\}$, and integers $i,j$ with $i<j<n$ such that $V_{1}=U_{i}$ and $V_{m}=U_{j}$, then we denote $\mathcal{V}$ by $\mathcal{U}(i,j)$. Said shortly, if $\mathcal{V}$ is a subchain of $\mathcal{U}$ with the first link of $\mathcal{V}$ being $U_{i}$ and the last link of $\mathcal{V}$ being $U_{j}$, then we denote $\mathcal{V}$ by $\mathcal{U}(i,j)$. We say that a simple chain $\mathcal{V}$ (strongly) {\bf refines} another simple chain $\mathcal{U}$ if every link (closure of a link) in $\mathcal{V}$ is contained in some link of $\mathcal{U}$. The final notion we need to construct for the pseudoarc is crookedness. 

\begin{definition}
Given a simple chain $\mathcal{V}$ that refines a simple chain $\mathcal{U}$ we say that $\mathcal{V}$ is {\bf crooked} in $\mathcal{U}$ if it satisfies the following:
\begin{itemize}
\item Whenever $\mathcal{V}(i,j)$ is a subchain of $\mathcal{V}$ with $V_{i}\cap U_{h}\neq\emptyset$ and $V_{j}\cap U_{k}\neq\emptyset$ where $|h-k|\geq 3$, there are $r$ and $s$ such that
\[\mathcal{V}(i,j)=\mathcal{V}(i,r)\cup\mathcal{V}(r,s)\cup\mathcal{V}(s,j)\]
\noindent
where $(s-r)(j-i)>0$ and $V_{r},V_{s}$ are subsets of links of $\mathcal{U}(h,k)$ that intersect $U_{k}$ and $U_{h}$, respectively.

\end{itemize}
\end{definition}

This definition is much easier to parse upon seeing the image on the next page. In this figure two simple chains between a pair of points in $\mathbb{R}^{2}$ are depicted. One is a simple chain of four elements of large diameter and the other is a simple chain of thirty elements of small diameter. For the smaller chain to be crooked in the large chain the smaller chain must "double back" on itself twice for every pair of elements of the large chain that are at least three links apart.
\vspace{\baselineskip}

\begin{figure}[!h]
\centering
\includegraphics{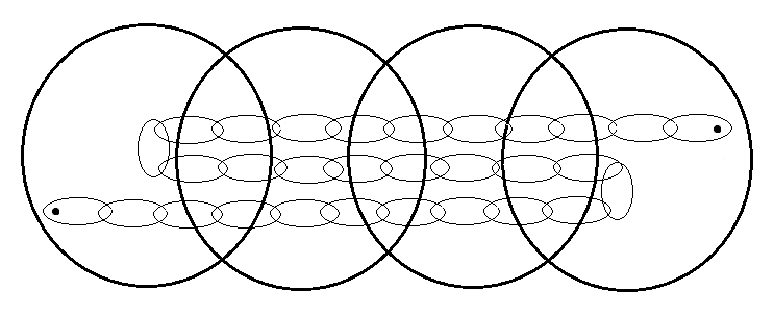}
\caption{A smaller chain that is crooked in a larger chain.}
\end{figure}

The pseudoarc is then constructed as a subset of $\mathbb{R}^{2}$ as follows. Let $p$ and $q$ be distinct points of $\mathbb{R}^{2}$ and let $\mathcal{C}_{1}$ be a simple chain between them made up of connected open subsets. Then, for $n\geq 2$ one can inductively construct a sequence of simple chains $\mathcal{C}_{n}$ between $p$ and $q$ made up of open connected subsets of $\mathbb{R}^{2}$ such that $\mathcal{C}_{n}$ strongly refines $\mathcal{C}_{n-1}$, $\mathcal{C}_{n}$ is crooked in $\mathcal{C}_{n-1}$, and the diameters of elements of $\mathcal{C}_{n}$ are at most $\frac{1}{n}$. The pseudoarc is the space $X=\bigcap_{n=1}^{\infty}\left(\bigcup\mathcal{C}_{n}\right)$ with the metric topology inherited from $\mathbb{R}^{2}$. 
\vspace{\baselineskip}

One final remark before proving our main result of this section. We moved to our discussion of indecomposable and hereditarily indecomposable continua after suggesting that the peculiarity of the Knaster-Kuratowski fan might just be a unique and uncommon oddity. The reader might be inclined to make the same judgement of indecomposable and hereditarily indecomposable continua as well. However, these kinds of continua are not at all uncommon. In fact, in a sense, "most" continua in $\mathbb{R}^{n}$ are hereditarily indecomposable.\footnote{One may denote the set of all continua in $\mathbb{R}^{n}$ by $C(\mathbb{R}^{n})$ and endow it with the Hausdorff metric $d_{h}$. Given two continua $X,Y\in C(\mathbb{R}^{n})$, $d_{h}(X,Y)$ is the infimum over all $\epsilon>0$ such that $X\subseteq B(Y,\epsilon)$ and $Y\subseteq B(X,\epsilon)$. Here $B(X,\epsilon)$ is the open metric ball about $X$ of radius $\epsilon$. The result we reference says that hereditarily indecomposable continua are a dense subset of this space.} What's more, contained in every $n+1$ dimensional subset of $\mathbb{R}^{n+1}$ is an $n$-dimensional hereditarily indecomposable continuum (where $n$-dimensional is with respect to the classical definitions). Proofs of these results can be found in \cite{nadlerHyperspace}. With that said, let's proceed to the main result of this section. 

\begin{theorem}\label{hindecomposableAbbott}
If $X$ is a hereditarily indecomposable continuum, then $Ab(X)\leq 1$, and $Ab(X)=0$ if and only if $X$ is a singleton.
\end{theorem}
\begin{proof}
We will use the following property of hereditarily indecomposable continua that follows from the definition, if $K$ is a hereditarily indecomposable continuum and $K_{1},K_{2}\subseteq K$ are subcontinua such that $K_{1}\cap K_{2}\neq\emptyset$ then either $K_{1}\subseteq K_{2}$ or $K_{2}\subseteq K_{1}$. Now let $X$ be a hereditarily indecomposable continuum. If $X$ is not a singleton, then because $X$ is a continuum we have that $Ab(X)\geq 1$. Assume $Y_{1}\ll_{k}Y_{2}$ for some $k$ with $z\in Y_{2}$ an observation point. If $x_{1},x_{2}\in Y_{1}$ are distinct points, then there are disjoint open neighborhoods $U_{1},U_{2}\subseteq Y_{1}$ containing $x_{1}$ and $x_{2}$ respectively. Then there is a continuum $W_{z,x_{1}}\subseteq Y_{2}$ that contains $x_{1}$ and $z$. Moreover, there is a connected open set $V_{1}\subseteq Y_{2}$ containing $W_{z,x_{1}}$ and such that $x_{1}\in cl_{Y_{2}}(V_{1})\cap Y_{1}\subseteq U_{1}$. Similarly there is a continuum $W_{z,x_{2}}$ and corresponding connected open set $V_{2}\subseteq Y_{2}$. However, as $z$ is an element of $W_{z,x_{1}}$ and $W_{z,x_{2}}$ we must have that $W_{z,x_{1}}\subseteq W_{z,x_{2}}$ or $W_{z,x_{2}}\subseteq W_{z,x_{1}}$. Then $cl_{Y_{2}}(V_{1})\cap Y_{1}\not\subseteq U_{1}$ and $cl_{Y_{2}}(V_{2})\cap Y_{1}\not\subseteq U_{2}$. This is a contradiction. Our beginning assumption was that $Y_{1}$ contained distinct points. We must then have that $Y_{1}$ is a singleton, and singletons are in $X_{0}$ and no other $X_{k}$. Therefore we have that $Ab(X)\leq 1$. 
\end{proof}

\section{Discussion.}
In the introduction we asked two main questions. First, what is dimension? This was not meant to be a question that we would actually answer, but one that would inspire us to reflect on dimension and create definitions for and descriptions of dimension, as Poincar\'e, Lebesgue, and the various other mathematicians mentioned in the introduction did. This led us to our second question, namely, are the definitions of classical dimension theory the "right" definitions? We said that the "right" definition could be characterized as a geometrically intuitive definition such that every other geometrically intuitive definition must agree with the "right" definition on the spaces we take to be natural. We have argued in this paper that the answer to this question is "no" by way of Abbott dimension. The reader may at this point raise multiple objections. They may disagree with our characterization of what the "right" definition of dimension should be. It is our hope that the reader with such an objection will come up with their own notion of what it means to have the "right" definition and ask themselves if the "right" definition of dimension has been found yet. The reader may also object to our positioning of Abbott dimension as being geometrically intuitive. This objection is very welcome (hoped for even) and to such a reader we would issue the challenge of finding their own definition of dimension that is, to them, geometrically intuitive and then compare that definition to the classical ones. Even if no definition is ultimately created, the time spent reflecting on how to express what dimension is is time well spent. 
\vspace{\baselineskip}

\noindent
{\bf Acknowledgement:} The author would like to thank Kate Franklin for her reading of, and input on, early drafts of this paper. The author would also like to thank the referees and editors for this paper whose input greatly improved it. This work was supported by the Israel Science Foundation grant No. 2196/20

\section{Appendix.} In this appendix we have collected results deemed technical enough to be distracting from the main narrative of the paper. They are none the less important. In particular, this appendix includes a discussion of the class of model spaces and surrounding technical results. This class is instrumental in proving the ever important Theorem \ref{euclidean space}, the fact that $Ab(\mathbb{R}^{n})=n$. We begin with a question that arises naturally after defining Abbott dimension. Given a space $X$ with $Ab(X)=n$, must it be the case that $X\in X_{n}$? It turns out that the answer to this is no, as the following example shows.

\begin{example}\label{need for model spaces}
Let $X$ be the union of the following subsets of $\mathbb{R}^{3}$: $\{0\}\times\{0\}\times\mathbb{R}$, $\{(q,r,s)\mid q,r,s\in\mathbb{Q},\,s\geq 0\}$, and $\{(x,y,0)\in\mathbb{R}^{3}\mid x^{2}+y^{2}\leq 1\}$. Equip $X$ with the metric inherited from $\mathbb{R}^{3}$. Then $X\in X_{1}$ because of the negative part of the $y$-axis. However, $X\notin X_{2}$ as no open subset of $X$ containing a point of the disc is connected. However, the disc itself is an element of $X_{2}$.  
\end{example}

We will find it useful to work with spaces (or subspaces) $X$ such that $Ab(X)=n$ and $X\in X_{n}$. One might think of these spaces as being the most "well behaved" with respect to Abbott dimension. We codify this property in the following definition.

\begin{definition}A space $X$ is a {\bf model} space if $Ab(X)=\sup\{n\in\mathbb{N}\mid X\in X_{n}\}$. Similarly, if $Y\subseteq X$ is a subspace such that $Ab(Y)=n$ and $Y\in X_{n}$ we call $Y$ a {\bf model subspace}.
\end{definition}

Forunately, in the case of spaces with finite Abbott dimension, such subspaces are always at hand.

\begin{proposition}\label{existence of model subspace}
If $X$ is a space with $Ab(X)=n$, then $X$ contains a model subspace $Y$ such that $Ab(Y)=n$. 
\end{proposition}
\begin{proof}
If $Ab(X)=n$, then $X_{n}\neq\emptyset$ so there is a $Y\subseteq X$ such that $Y\in X_{n}$. For every subspace $K\subseteq Y$ we have that $K\in X_{m}$ for some $m\leq n$ because $Ab(X)=n$. Therefore $Y$ is a model subspace. 
\end{proof}
\begin{proposition}
If $(X,d)$ is a nonempty metric space such that $Ab(X)=0$, then $X$ is a model space. 
\end{proposition}

Recall that when discussing the fact that $Ab(\mathbb{R}^{n})\geq n$ we mentioned a particular sequence of subspaces. Specifically, for $\mathbb{R}^{3}$ we mentioned that $Ab(\mathbb{R}^{3})\geq 3$ could be seen by considering a point, a line segment containing that point, a square one of whose sides is that segment, and a cube one of whose faces is the square. Nested sequences of subspaces of this kind turn out to be quite useful, even more so when the subspaces concerned are model subspaces. We codify sequences of such subspaces in the following definition. 

\begin{definition}
A {\bf dimensional profile} or simply {\bf profile} of a space $X$ with $Ab(X)=n$ is a sequence of subspaces $Y_{0},\ldots,Y_{n}$ such that $Y_{0}\ll_{1}Y_{1}\ll_{2}\cdots\ll_{n}Y_{n}$. A profile $Y_{0},\ldots,Y_{n}$ for a space $X$ is called a {\bf model profile} if $Y_{n}=X_{n}$ and $Y_{k}$ is a model space for each $k$. Given a property $P$, we will refer to a profile $Y_{1},\ldots,Y_{n}$ for a metric space $X$ as a {\bf $P$ profile} for $X$ if every $Y_{k}$ has property $P$ (e.g. a connected profile). If $Ab(X)=\infty$ we define a profile to be an infinite sequence of subspaces $(Y_{i})_{i\in\mathbb{N}}$ such that $Y_{i}\ll_{i+1}Y_{i+1}$ for each $i$. 
\end{definition}

\begin{proposition}\label{model profile}
Every model space $X$ of finite Abbott dimension admits a model profile.
\end{proposition}
\begin{proof}
Let $X$ be a given model space with $Ab(X)=n$. Then there is an element $Y\in X_{n-1}$ such that $Y\ll_{n}X$. As $Y\subseteq X$ we have that $Ab(Y)\leq n$, and so by Proposition \ref{existence of model subspace} there must be a model subspace $Y_{n-1}\subseteq Y$ with the same Abbott dimension as $Y$ which is at least $n-1$. Then $Y_{n-1}\in X_{n-1}$. Because $Y$ witnesses $X$ being in $X_{n}$ and $Y_{n-1}$ is a subset of $Y$ we have that $Y_{n-1}$ witnesses $X$ being in $X_{n}$ as well by Lemma \ref{subspaces witness}. We may repeat this process to find $Y_{k}$ for $1\leq k\leq n-2$, yielding a model profile for $X$.

\end{proof}

In the following theorem we show that the class of model spaces includes some well-known spaces, specifically locally connected spaces. This may not come as a shock given how much the definition of Abbott dimension is tied to connected subsets. Recall that if $X$ is a locally connected space with subspace $Y$ and $U\subseteq Y$ is a connected open subset of $Y$, then there is some connected open $\hat{U}\subseteq X$ such that $\hat{U}\cap Y=U$. To see this we need only recall that in a locally connected space, the connected components of every open set are open.

\begin{theorem}
If $(X,d)$ is locally connected then $X$ is a model space. 
\end{theorem}
\begin{proof}
Assume that $X\in X_{m}$, $X\notin X_{m+1}$, and $Y\subseteq X$ is such that $Y\in X_{n}$ for some $n>m$. Let $Y_{2}\in X_{n-1}$ witness $Y\in X_{n}$. If $V\subseteq Y$ is a connected open set used to show $Y\in X_{n}$, then by the local connectedness of $X$ we have that $U=\hat{V}\cap Y$ for some connected open $\hat{V}\subseteq X$. As $\partial_{Y}V\subseteq\partial_{X}\hat{V}$ we have that if $Y\in X_{n}$, then so must be $X$, contradicting $X\notin X_{m+1}$. Therefore, $X$ is a model space. 
\end{proof}

\begin{corollary}
$\mathbb{R}^{n}$ is a model space for every $n\geq 1$.
\end{corollary}

With these results established we can finally prove that $Ab(\mathbb{R}^{n})=n$. As stated in the main body of the paper, we consider the finite and infinite dimensional cases separately.

\begin{theorem}\label{Abbott less than Ind finite}
Let $X$ be a separable model metric space with $Ab(X)=n<\infty$. Then $Ind(X)\geq n$.
\end{theorem}
\begin{proof}
As $Ab(X)=0$ implies that $X\neq\emptyset$ we have that $Ab(X)=0$ implies that $Ind(X)\geq 0$. If $Ab(X)=1$, then $X$ necessarily contains a nontrivial connected set (some $W_{z,x}$ as in the definition of $Ab(X)$) and $Ind(X)=0$ implies $X$ is totally disconnected, so $Ind(X)\geq 1$. Now assume that the result holds for $Ab(X)<n$ where $n\geq 2$ and assume towards a contradiction that $Ab(X)=n$ and $Ind(X)=m<n$ where $m\geq 0$ (note that $Ind(X)=-1$ yields an immediate contradiction as $X\neq\emptyset$). Then let $Y_{0},Y_{1},\ldots,Y_{n-1},X$ be a model profile for $X$. Let $z\in X\setminus Y_{n-1}$ be an observation point for $Y_{n-1}$. Let $x\in Y_{n-1}$ and let $W_{z,x}\subseteq X$ be a line of sight from $z$ to $x$ as in the definition of $Ab(X)$. Recall that the collection $\mathcal{C}(x,z,X)$ must contain a neighborhood basis for $W_{z,x}$. By the definition of $Ind$, there is an open set $V\subseteq X$ containing $W_{z,x}$ such that for all open sets $\hat{V}\subseteq X$ with $W_{z,x}\subseteq\hat{V}\subseteq V$ we have that $Ind(\partial_{X}\hat{V})\leq m-1<n-1$. By the definition of $Ab(X)$ we must have that there is some smaller neighborhood $V^{\prime}\in\mathcal{C}(z,x,X)$ containing $W_{z,x}$ and contained in $\hat{V}$ such that if $V^{\prime\prime}\in\mathcal{C}(x,z,X)$ contains $W_{z,x}$ and is contained in $V^{\prime}$ then there is some $D\subseteq\partial_{X}(V^{\prime\prime})$ such that $D\in X_{n-1}$. Letting $V^{\prime\prime}\subseteq X$ with $D\subseteq\partial_{X}(V^{\prime})$ be such sets we must have that $Ind(\partial_{X}(V^{\prime\prime}))\leq m-1<n-1$. As $D\in X_{n-1}$ we have that $Ab(D)\geq n-1$. Because $Ab(X)$ is finite, we also have that $Ab(D)$ is finite by Theorem \ref{subspace theorem}. Then by Proposition \ref{intermediate dimensions} there must be some $H\subseteq D\subseteq\partial_{X}(V^{\prime\prime})$ such that $Ab(H)=n-1$. However, we then must have that $Ind(H)\leq m-1$ by Theorem \ref{subspaceInd}. This contradicts our inductive hypothesis. Therefore $Ind(X)\geq n$. 
\end{proof}

\begin{theorem}
Let $X$ be a separable metric space with $Ab(X)=\infty$. Then $Ind(X)=\infty$. 
\end{theorem}
\begin{proof}
For each integer $n\geq 0$ consider the following statement.

\begin{equation}\label{eqn}
Z \text{ is a separable metric space such that } Ab(Z)=\infty\text{ and }Ind(Z)=n.\tag{$*_n$}
\end{equation}

\noindent
We will prove that (\ref{eqn}) does not hold for any $n$ by induction on $n$. First, for $n=0$ we have that if $Ind(X)=0$ then $X$ is totally disconnected, but $Ab(X)=\infty$ implies that $X_{1}\neq\emptyset$ which implies that $X$ contains a nontrivial connected subset, a contradiction. Therefore, $(*_0)$ does not hold. Assume that $(*_m)$ does not hold for $0\leq m<n$ and assume towards a contradiction that $(*_n)$ holds. Because $Ab(X)=\infty$ we have that $X_{k}\neq\emptyset$ for all nonnegative integers $k$. We then let $Y_{3n}\in X_{3n}$. We can then find a set $Y_{3n-1}\in X_{3n-1}$ such that $Y_{3n-1}\ll_{3n}Y_{3n}$. If $Ab(Y_{3n})<\infty$ then $Ab(Y_{3n})\geq 3n$ and by the conjunction of Theorem \ref{subspaceInd}, Theorem \ref{Abbott less than Ind finite}, and Proposition \ref{existence of model subspace} we would have that $Ind(Y_{3n})\geq 3n$. However this would be a contradiction because $Ind(X)=n$. Therefore $Ab(Y_{3n})$ must be infinite. By an identical argument we must have that any element of $X_{k}$ for $k>n$ must have infinite Abbott dimension. Now let $z\in Y_{3n}\setminus Y_{3n-1}$ be an observation point for $Y_{3n-1}$, $x\in Y_{3n-1}$, and let $W_{z,x}$ be a continuum in $Y_{3n}$ connecting $z$ and $x$ as in the definition of Abbott dimension. Because $Y_{3n}\in X_{3n}$ there is some $V\in\mathcal{C}(z,x,Y_{3n})$ containing $W_{z,x}$ such that if $V^{\prime}\in\mathcal{C}(z,x,Y_{3n})$ is another element containing $W_{z,x}$ then there must be some $D\subseteq\partial_{Y_{3n}}(V^{\prime})$ such that $D\in X_{3n-1}$. Because $Ind(X)=n$ and $\mathcal{C}(x,z,Y_{3n})$ contains a basis for $W_{z,x}$ we have that $V$ can be chosen such that every $V^{\prime}\in\mathcal{C}(z,x,Y_{3n})$ contained in $V$ and containing $W_{z,x}$ is such that $Ind(\partial_{Y_{3n}})\leq n-1$. However, this produces a contradiction as if $V\in\mathcal{C}(z,x,Y_{3n})$ and $V^{\prime}\in\mathcal{C}(z,x,Y_{3n})$ is contained in $V$ and contains $W_{z,x}$ then there must be a $D\subseteq\partial_{Y_{3n}}(V^{\prime})$ such that $D\in X_{3n-1}$, but the preceeding discussion would imply that both $Ab(D)=\infty$ and $Ind(D)\leq n-1$ which contradicts our inductive assumption that $(*_k)$ does not hold for any separable metric space for $k<n$. We have then shown that $(*_n)$ can not hold for any nonnegative integer $n$ and therefore we must have that $Ind(X)=\infty$.

\end{proof}

\begin{corollary}
If $X$ is a separable metric space with $Ab(X)<\infty$ then $Ab(X)\leq Ind(X)$.
\end{corollary}
\begin{proof}
If $X$ is a separable metric space with $Ab(X)=n$, then there is a model subspace $Y\subseteq X$ with $Ab(Y)=n$ by Proposition \ref{existence of model subspace}. By Theorem \ref{Abbott less than Ind finite} we must have that $Ab(Y)\leq Ind(Y)$ and by Theorem \ref{subspaceInd}, $Ind(Y)\leq Ind(X)$. Then $Ab(X)=Ab(Y)\leq Ind(Y)\leq Ind(X)$. 
\end{proof}

\noindent
Theorem \ref{euclidean space} then follows as a quick corollary.

%\bibliography{AbbottDimension} \bibliographystyle{plain}

\end{document}